\documentclass[12pt,leqno]{article}
\usepackage{amsmath,amstext, amsthm,amsfonts,amssymb,amsthm}
 \numberwithin{equation}{section}
\def\R{\mathbb{ R}}  
  
\usepackage{epsfig}
\newtheorem{thm}{Theorem}[section]
\newtheorem{lemma}[thm]{Lemma}

\newtheorem{rem}[thm]{Remark}
\newtheorem{prop}[thm]{Proposition}
\newcommand{\be}{\begin{equation}}
\newcommand{\ee}{\end{equation}}

\newcommand{\ba}{\begin{array}}
\newcommand{\ea}{\end{array}}

\newcommand{\al}{\alpha}
\newcommand{\bt}{\beta}

\newcommand{\la}{\lambda}

\newcommand{\bg}{\begin{gathered}}
\newcommand{\eg}{\end{gathered}}

\newcommand{\bea}{\begin{eqnarray}}
\newcommand{\eea}{\end{eqnarray}}

\title{On the Atkin Polynomials}
\author{Ahmad El-Guindy \and   Mourad E.H. Ismail\thanks{Research supported by the DSFP at King Saud University through grant DSFP/MATH 01 and by the  NPST program of King Saud University; project number 10-MAT1293-02.}  }

\begin{document}
 \maketitle
 \date{}
\begin{abstract}
 We identify the Atkin polynomials in terms of associated Jacobi polynomials. Our identification 
 then takes advantage of the theory of orthogonal polynomials and their asymptotics to establish many new properties of the Atkin polynomials. This shows that co-recursive  polynomials may lead to interesting sets of orthogonal polynomials. 
  \end{abstract}






\section{Introduction}
In unpublished work Oliver Atkin introduced a family of orthogonal polynomials with fascinating 
 number theoretic properties: They are the unique family of monic orthogonal polynomials 
  corresponding to a unique scalar product on the space of polynomials in the modular 
    $j$-invariant for which all Hecke operators 
are self-adjoint. Furthermore, their reduction modulo a prime $p\geq 5$ is also very significant 
 in the theory of elliptic curves, as it matches the supersingular polynomial at $p$ whenever the 
  degrees agree.
For all the number theoretic definitions, as well as beautiful proofs of these and other 
 facts about the Atkin polynomials, we refer the reader to the excellent paper 
  \cite{Kan:Zag} by Kaneko and Zagier, where Atkin's results were popularized, 
   simplified, and expanded upon.

The Atkin polynomials are generated by  the recurrence relation 
 \bea\label{rec1}
 \begin{gathered}
 A_{n+1}(x) = \left[x -  24 \frac{144 n^2- 29}{(2n+1)(2n-1)}\right]A_n(x)    \\
 - 36 \frac{(12n -13)(12n -7)(12n-5)(12n+1)} {n(n-1)(2n-1)^2} A_{n-1}(x), n >1. 
 \end{gathered}
 \label{eqAtkinrr}
 \eea
through the  initial conditions 
 \bea
 \label{eqAtkinIC}
 A_0(x) = 1, \quad  A_1(x) = x - 720, \quad A_2(x) = x^2 - 1640 x + 269280. 
 \eea
 The polynomials $\{A_n(x)\}$ are orthogonal with respect to an absolutely continuous 
  measure supported on $[0, 1728]$ (cf. Section 7). 

  
  In this paper we show that the Atkin polynomials are related to the associated Jacobi polynomials 
  of Wimp \cite{Wim} and of  Ismail and Masson \cite{Ism:Mas}. 
   This identification leads to many new properties of 
  the polynomials $\{A_n(x)\}$.

 It is worth pointing out that the way the Atkin polynomials are
defined, that is define $P_0(x),
   P_1(x)$ and $P_2(x)$, then use a recurrence relation to generate
the rest, is not unusual
    in the literature on orthogonal polynomials.  The idea is to start
with two monic polynomials,  $P_k(x)$ and $P_{k+1}(x)$
  of degrees $k$ and $k+1$, respectively,
with real, simple and
interlacing zeros. Then use the
division algorithm  to generate the
 monic  polynomials $P_{n}(x), 0 \le n < k$ and we are guaranteed to
have a sequence of monic
 orthogonal polynomials $\{P_j(x): 0 \le j \le k+1\}$. Now use any
 three term recurrence
relation of the form 
\bea
\label{eqmonic}
P_{n+1}(x) = (x-\al_n)P_n(x) + \beta_n P_{n-1}(x), 
\eea
 where $\al_n \in \R$ and $\bt_n >0$,  for $n >k$ to generate the polynomials $\{P_n(x)\}$ for $n > k+1$. 
  The construction above is referred to as
``Wendroff's Theorem" in the orthogonal
polynomial literature. The interested reader may consult \cite{Ism} or
\cite{Chi}
  for the precise statement and the detailed
 proof of Wendroff's theorem. This is also related to the concept of co-recursive polynomials, \cite{Chi}. 

In Section 2 we recall some preliminary facts about associated Jacobi polynomials and orthogonal polynomials in general. In Section 3, we obtain a representation of (a scaled version of ) the Atkin polynomials as a linear combination of the associated Jacobi polynomials of Wimp \cite{Wim} and of Ismail and Masson \cite{Ism:Mas}. Building on that, we provide an explicit representation of the coefficients of the Atkin polynomials in Section 4, a representation in terms of certain 
hypergeometric functions and an asymptotic expansion in Section 5, and a generating function identity in Section 6. Lastly in Section 7 we give an explicit description of the weight function in terms of certain $\, _2F_1$ functions.  
  
 We shall follow the standard notation for hypergemetric functions and orthogonal polynomials as 
 in \cite{And:Ask:Roy}, \cite{Ism}, \cite{Luk}, \cite{Rai}, \cite{Sze}.   In particular we use $F(a,b;c;z)$ to mean ${}_2F_1(a,b;c;z)$. 
  
  \section{Preliminaries}

Let $\{\la_n\}$ and $\{\mu_n\}$ be the birth and death rates of a birth and death process, 
that is $\la_n >0,$ 
and $\mu_{n+1} > 0$ for all $n \ge 0$ with $\mu_0 \ge 0$. Such process generates  a sequence of orthogonal polynomials through a three term recurrence relation 
\bea
\label{eqBDPP}
-x Q_n(x) = \la_n Q_{n+1}(x) - (\la_n +\mu_n) Q_n(x) + \mu_nQ_{n-1}(x), \quad n >0, 
\eea
with the initial conditions 
\bea
\label{eqIC} 
Q_0(x) =1, \qquad Q_1(x) = (\la_0+\mu_0-x)/\la_0.
\eea
  The corresponding monic polynomials satisfy
\bea
\label{eqmonicBD}
x\tilde{Q}_n(x) = \tilde{Q}_{n+1}  + (\la_n+\mu_n)\tilde{Q}_n(x) - \la_{n-1} \mu_n \tilde{Q}_{n-1}(x), 
\eea
with $\tilde{Q}_0(x) =1, \tilde{Q}_1(x) = x-\la_0-\mu_0.$ When $\mu_0 \ne 0$ there is a 
second natural birth and death process with birth rates $\{\la_n\}$ and death rates $\{\tilde{\mu}_n\}$ with 
 $\tilde{\mu}_n = \mu_n, n>0$ but   $\tilde{\mu}_0 =0$ \cite{Ism:Let:Val}. 
  The latter birth and death generates a  second family of orthogonal polynomials 
   satisfying \eqref{eqBDPP} but with initial conditions 
  $Q_0(x) =1, Q_1(x) = (\la_0 -x)/\la_0.$ This observation is due to Ismail, 
  Letessier and Valent  in \cite{Ism:Let:Val}. 
   
The associated polynomials of $\{Q_n(x)\}$ correspond to the birth and death rates 
$\{\la_{n+c}\}$ and death rates $\{\mu_{n+c}\}$, when such rates are well defined.  Since we 
consider $c \ge 0$, usually $\mu_c>0$.  Thus  we usually have two families 
 of associated  polynomials. 
 One is defined  when $\mu_c$ is defined from the pattern of $\mu_{n}$. When $\mu_c \ne 0$, a second 
 family arises if $\mu_{n+c}$, when $n=0$ is interpretted   as zero.  
 
 Recall that the Jacobi polynomials $\{P_n^{(\al, \beta)}(x)\}$ can be defined by the  three term recurrence relation  
 \bea
\label{eqttrrJ}
\begin{gathered}
2(n+1)(n+\alpha+\beta+1) (\alpha+\beta+2n) P_{n+1}^{(\alpha, \beta)}(x)  \qquad \\
 \qquad =(\alpha+\beta+2n+1)\left[(\alpha^2-\beta^2)+
 x(\alpha+\beta+2n+2)(\alpha+\beta+2n)\right]  \\
 \qquad  \times P_{n}^{(\alpha, \beta)}(x)-2(\alpha+n)(\beta+n)
 (\alpha+\beta+2n+2) P_{n-1}^{(\alpha,\beta)}(x),
\end{gathered}
\eea
for $n\ge 0$, with $P_{-1}^{(\alpha,\beta)}(x)=0$, $P_0^{(\alpha,\beta)}(x)=1$.  We now set 
\bea
\label{eqdefV}
\qquad  V^{(\al,\beta)}_n(x) 
=  \frac{n!(\al+\bt+1)_n}{(\al+ \beta+1)_{2n}} P_n^{(\al,\beta)}(2x-1) 
=  \frac{n!}{(n+\al+ \beta+1)_n} P_n^{(\al,\beta)}(2x-1). 
\eea
One can easily verify that the polynomials $\{V^{(\al,\beta)}_n(x)\}$ are 
  monic  birth and death process polynomial $\tilde{Q}_n$'s,  
 with rates 
 \bea
 \bg 
 \la_n  = \frac{(n+\beta+1)(n+\al+\beta+1)}{(2n+\al+\beta+1)(2n+\al+\beta+2)},\\
 \mu_n = \frac{n(n+\al)}{(2n+\al+\beta)(2n+\al+\beta+1)}. 
 \eg
 \label{eqBDrforJ}
 \eea
In \cite{Wim}, Wimp considered the recurrence relation obtained by formally replacing $n$ by $n+c$ in \eqref{eqttrrJ}, and he showed that the new relation has two linearly independent solutions $P_n^{(\al,\beta)}(x;c)$ and $P_{n-1}^{(\al,\beta)}(x;c+1)$. Ismail and Masson \cite{Ism:Mas} 
 identified the birth and death rates corresponding to that three term recurrence relation and provided two linearly independent solutions 
 $P_n^{(\al,\beta)}(x;c)$ and  $\mathcal{P}_n^{(\al,\beta)}(x;c)$. They then used the notation 
 \bea
 R_n^{(\al,\beta)}(x;c) = P_n^{(\al,\beta)}(2x-1;c),  \quad 
 \mathcal{R}_n^{(\al,\beta)}(x;c)= \mathcal{P}_n^{(\al,\beta)}(2x-1;c). 
 \eea
 We shall use the notation 
 \bea
\label{eqdefVc}
\bg
V^{(\al,\beta)}_n(x;c) =\frac{(c+1)_n(\al+\bt+c+1)_n}{(\al+\beta+2c+1)_{2n}} 
R_n^{(\al.\beta)}(x;c),  \\
\mathcal{V}^{(\al,\beta)}_n(x;c) =\frac{(c+1)_n(\al+\bt+c+1)_n}{(\al+\beta+2c+1)_{2n}}
 \mathcal{R}_n^{(\al.\beta)}(x;c). 
\eg
\eea
To lighten our notation, we shall occasionally omit the parameters when the context is clear.
We consider the  birth and death rates  
\bea
 \bg 
 \la_n  = \frac{(n+c+\beta+1)(n+c+\al+\beta+1)}{(2n+2c+\al+\beta+1)(2n+2c+\al+\beta+2)}, \quad n \ge 0,\\
 \mu_n = \frac{(n+c)(n+c+\al)}{(2n+2c+\al+\beta)(2n+2c+\al+\beta+1)}, \quad n>0,
 \eg
 \label{eqBDrforAJ}
 \eea
with 
\bea
\label{eqdefmu0}
\mu_0 :=  \begin{cases} & \frac{c(c+\al)}{(2c+\al+\beta)(2c+\al+\beta+1)} \quad \textup{for} \; V, \\
& 0  \quad \textup{for}\; \mathcal{V}.
\end{cases} 
\eea

\section{The Atkin Polynomials} 
 In order to compare the Atkin  polynomials with other 
   results in the literature we need to renormalize them. Let 
 \bea
 \label{eqAnBn}
 A_n(1728 y) = (1728)^n \mathcal{A}_n(y). 
 \eea
 The polynomials $\mathcal{A}_n$ are now generated by 
 \bea
 \begin{gathered}
 \mathcal{A}_{n+1}(x) = \left[x - \frac{2(n^2-29/144)}{4n^2-1}\right] \mathcal{A}_n(x)  \qquad \qquad \qquad \qquad  \\
\qquad \qquad  - \frac{(n-13/12)(n- 7/12)(n-5/12)(n+1/12)}
 {2n(2n-1)^2(2n-2)} \mathcal{A}_{n-1}(x), \quad n >1. 
  \end{gathered}
  \eea
 The initial conditions are 
 \bea\label{Binitial}
 \mathcal{A}_0(x) =1,  \quad  \mathcal{A}_1(x) = x-5/12,  \quad \mathcal{A}_2(x) = x^2 - \frac{205}{216} x + \frac{935}{10368}. 
 \eea
Kaneko  and  Zagier \cite{Kan:Zag}  wrote the recurrence relation \eqref{eqAtkinrr} in 
 the monic form \eqref{eqmonicBD}. 
   Indeed their (19) when written in terms of the $\mathcal{A}_n$'s corresponds to \eqref{eqmonicBD} with 
 \bea
 \la_n = \frac{(n-1/12)(n+5/12)}{2n(2n+1)}, \quad \mu_n = \frac{(n-5/12)(n+1/12)}{2n(2n-1)}.
 \eea   
From \eqref{eqdefVc} we see that $V_n^{(\alpha, \beta)}(x;c)$ and $\mathcal{V}_n^{(\alpha, \beta)}(x;c)$ satisfy the second order difference equation  
\bea
\bg
\label{Vrec}
T_{n+1}(x)=\left(x+\frac{\alpha^2-\beta^2 -(2n+2c+\alpha+\beta)(2n+2c+\alpha+\beta+2)}{2(2n+2c+\alpha+\beta)(2n+2c+\alpha+\beta+2)}\right) T_n(x)\\
-\frac{(n+c)(n+c+\alpha)(n+c+\beta)(n+c+\alpha+\beta)}{(2n+2c+\alpha+\beta-1)(2n+2c+\alpha+\beta)^2(2n+2c+\alpha+\beta+1)}T_{n-1}(x), 
\eg
\eea
for $n\geq 1$ with the  initial conditions  $V_0=\mathcal{V}_0=1$ and 
\bea\label{Vinitial}
\bg
V^{(\al,\beta)}_1(x;c) = x-(\lambda_0+\mu_0)  \\
\mathcal{V}^{(\al,\beta)}_1(x;c) =x-\lambda_0, 
\eg
\eea
where $ \la_n$ and $\mu_n$ are defined as in \eqref{eqBDrforAJ}--\eqref{eqdefmu0}. 
On the other hand, we see that the sequence  $\{\mathcal{A}_{n+1}(x)\}_{n=-1}^\infty$ is a solution of the 
second order difference equation 
\bea
\bg\label{Brec}
T_{n+1}(x)=\left(x-\frac{7+36(2n+1)(2n+3)}{72(2n+1)(2n+3)}\right) T_n(x)\\
-\frac{(n-1/12)(n+5/12)(n+7/12)(n+13/12)}{(2n)(2n+1)^2(2n+2)}T_{n-1}(x), \, (n\geq 1).
\eg
\eea
It is not hard to check that \eqref{Brec} is identical  to \eqref{Vrec} in exactly four cases, namely 
\[
(\alpha, \beta, c)\in S:=\left\{\left(\frac{-1}{2}, \frac{-2}{3}, \frac{13}{12}\right), \left(\frac{1}{2}, \frac{-2}{3}, \frac{7}{12}\right),\left(\frac{-1}{2}, \frac{2}{3}, \frac{5}{12}\right),\left(\frac{1}{2}, \frac{2}{3}, \frac{-1}{12}\right)\right\}.
\] 

\begin{thm}\label{thmBrep}
For $n\geq 0$ and $(\alpha, \beta, c)\in S$,  we have the following representations for $\mathcal{A}_{n+1}(x)$. 
\bea
\mathcal{A}_{n+1}(x)=(x-5/12)V_n^{(\alpha,\beta)}(x;c)-\frac{91}{384}V_{n-1}^{(\alpha, \beta)}(x;c+1),\\
\mathcal{A}_{n+1}(x)=(x-8)V_n^{(-1/2,2/3)}(x;5/12)+\frac{91}{12}\mathcal{V}_{n}^{(-1/2, 2/3)}(x;5/12),\\
\mathcal{A}_{n+1}(x)=xV_n^{(1/2,-2/3)}(x;7/12)-\frac{5}{12}\mathcal{V}_{n}^{(1/2,- 2/3)}(x;7/12).\label{Brep}
\eea
\end{thm}
\begin{proof}
It is straightforward to check that for any $(\alpha, \beta, c)\in S$, $\{V_n^{(\alpha, \beta)}(x;c),\mathcal{V}_{n}^{(\alpha, \beta)}(x;c)\}$ is a basis of solutions of \eqref{Brec}, and the same it true for $\{V_n^{(\alpha, \beta)}(x;c),V_{n-1}^{(\alpha, \beta)}(x;c+1)\}$. The results follow by simple linear algebra on the equations corresponding to $n=0$ and $n=1$.
\end{proof}

We note that $V_n^{(\alpha, \beta)}(x;c)$ is the same for the four triples in $S$. Whereas we have two possibilities for $\mathcal{V}_n^{(\alpha, \beta)}(x;c)$ depending on whether $\beta=2/3$ or $\beta=-2/3$. For convenience  we explicitly  write down the first few of these polynomials.

\bea\label{Vconcrete1}
\bg
V_0^{(\alpha,\beta)}(x;c)=1, \quad 
V_1^{(\alpha,\beta)}(x;c)=x-\frac{115}{216}, \\
V_2^{(\alpha,\beta)}(x;c)=x^2-\frac{187}{180}x+\frac{11621}{55296}, \\
\eg
\eea

\bea\label{Vconcrete2}
\bg
V_{-1}^{(\alpha,\beta)}(x;c+1)=0, \quad 
V_{0}^{(\alpha,\beta)}(x;c+1)=1, \\
V_{1}^{(\alpha,\beta)}(x;c+1)=x-\frac{547}{1080}, \\
\eg
\eea

\bea\label{Vconcrete3}
\bg
\mathcal{V}_0^{(1/2,-2/3)}(x;7/12)=\mathcal{V}_0^{(-1/2,-2/3)}(x;13/12)=1, \\
\mathcal{V}_1^{(1/2,-2/3)}(x;7/12)=\mathcal{V}_1^{(-1/2,-2/3)}(x;13/12)=x-\frac{187}{864}, \\
\mathcal{V}_2^{(1/2,-2/3)}(x;7/12)=\mathcal{V}_2^{(-1/2,-2/3)}(x;13/12)=x^2-\frac{347}{480}x+\frac{124729}{2488320},   \\
\eg
\eea

\bea\label{Vconcrete4}
\bg
\mathcal{V}_0^{(1/2,2/3)}(x;-1/12)=\mathcal{V}_0^{(-1/2,2/3)}(x;5/12)=1, \\
\mathcal{V}_1^{(1/2,2/3)}(x;-1/12)=\mathcal{V}_1^{(-1/2,2/3)}(x;5/12)=x-\frac{475}{864}, \\
\mathcal{V}_2^{(1/2,2/3)}(x;-1/12)=\mathcal{V}_2^{(-1/2,2/3)}(x;5/12)=x^2-\frac{169}{160}x+\frac{108965}{497664}. \\
\eg
\eea

One can check the first few cases of Theorem \ref{thmBrep} using the fact that
\bea
\bg\label{B123}
\mathcal{A}_1(x)=x-\frac{5}{12},\qquad 
\mathcal{A}_2(x)=x^2-\frac{205}{216}x+\frac{935}{10368},\\
\mathcal{A}_3(x)=x^3-\frac{131}{90}x^2+\frac{28277}{55296}x-\frac{124729}{5971968}.
\eg
\eea

\section{Explicit Representations} 
Wimp gave an explicit formula for  $R_n^{(\al,\beta)}(x;c)$ on page 
987 of \cite{Wim}. When translated in terms of the $V_n$ polynomials it becomes  
\bea
\bg
V_n^{(\al,\beta)}(x;c) = (-1)^n \frac{(c+1)_n(\beta+c+1)_n}{(\al+\beta+2c+n+1)_n \, n!} \\
\times \sum_{k=0}^n\frac{(-n)_k(n+2c+\al+\beta+1)_k}{(c+1)_k(c+\beta+1)_k} x^k \\
\times {}_4F_3 \left(\left. \begin{array}{c}
  k -n, n+k+\al+\beta+2c+1, c+\beta, c  \\
 k+\beta+ c+1, k+c+1, \al+\beta+2c
     \end{array} \right| 1  \right).
\eg
\eea
On the other hand Ismail and Masson \cite[Theorem 3.3]{Ism:Mas} gave a similar formula for  
$\mathcal{R}_n^{(\al,\beta)}(x;c)$ which leads to
\bea
\bg
\mathcal{V}_n^{(\al,\beta)}(x;c) = (-1)^n \frac{(c+1)_n(\beta+c+1)_n}{(\al+\beta+2c+n+1)_n \, n!} \\
\times \sum_{k=0}^n\frac{(-n)_k(n+2c+\al+\beta+1)_k}{(c+1)_k(c+\beta+1)_k} x^k \\
\times {}_4F_3 \left(\left. \begin{array}{c}
  k -n, n+k+\al+\beta+2c+1, c+\beta+1, c  \\
 k+\beta+ c+1, k+c+1, \al+\beta+2c+1
     \end{array} \right| 1  \right).
\eg
\eea
The following theorem establishes an analogous representation of $\mathcal{A}_{n+1}(x)$.
\begin{thm}
For $n\geq 0$ we have
\bea\label{ourrep}
\bg
\mathcal{A}_{n+1}(x)=\frac{(19/12)_n (11/12)_n}{(n+2)_n(-n)_n}\\
\times\left[{}_3F_2 \left(\left. \begin{array}{c}
   -n, n+2, \frac{7}{12}  \\
  \frac{19}{12}, 2
 \end{array} \right| 1  \right)
+\sum_{k=0}^n \frac{(-n)_k(n+2)_k}{(19/12)_k(11/12)_k}x^{k+1}\right.\\
\times\left.\left\{
\frac{6}{5}{}_4F_3 \left(\left. \begin{array}{c}
  k -n, n+k+2, \frac{11}{12}, \frac{-5}{12}  \\
 k+\frac{11}{12}, k+\frac{19}{12}, 1
     \end{array} \right| 1  \right)-\frac{1}{5}{}_4F_3 \left(\left. \begin{array}{c}
  k -n, n+k+2, \frac{-1}{12}, \frac{-5}{12}  \\
 k+\frac{11}{12}, k+\frac{19}{12}, 1
     \end{array} \right| 1  \right)\right\}\right]\\
\eg
\eea
\end{thm}

\begin{proof}
From \eqref{Brep} we have
\[
\mathcal{A}_{n+1}(x)=x{V}_{n}^{\left(\frac{1}{2}, \frac{-2}{3}\right)}\left(x;\frac{7}{12}\right)-\frac{5}{12}\mathcal{V}_{n}^{\left(\frac{1}{2}, \frac{-2}{3}\right)}\left(x;\frac{7}{12}\right),\, (n\geq 0), 
\]
we see that the coefficient of $x^{k+1}$ in $\mathcal{A}_{n+1}(x)$ is given by
\bea\label{Bcoeffx}
\bg
(-1)^n \frac{(19/12)_n(11/12)_n}{(n+2)_n \, n!} \\
\times \frac{(-n)_k(n+2)_k}{(19/12)_k(11/12)_k}
\times \left[{}_4F_3 \left(\left. \begin{array}{c}
  k -n, n+k+2, \frac{-1}{12}, \frac{7}{12}  \\
 k+\frac{11}{12}, k+\frac{19}{12}, 1
     \end{array} \right| 1  \right)\right.\\
\left.-\frac{5}{12}\frac{(k-n)(n+k+2)}{(k+19/12)(k+11/12)}{}_4F_3 \left(\left. \begin{array}{c}
  k+1 -n, n+k+3, \frac{11}{12}, \frac{7}{12}  \\
 k+1+\frac{11}{12}, k+1+\frac{19}{12}, 2
     \end{array} \right| 1  \right)\right].
\eg
\eea

The coefficient of $y^m$ in 
\begin{eqnarray*}
\bg
 \left[{}_4F_3 \left(\left. \begin{array}{c}
  k -n, n+k+2, \frac{-1}{12}, \frac{7}{12}  \\
 k+\frac{11}{12}, k+\frac{19}{12}, 1
     \end{array} \right| y  \right)\right.\\
\left.-\frac{5y}{12}\frac{(k-n)(n+k+2)}{(k+19/12)(k+11/12)}{}_4F_3 \left(\left. \begin{array}{c}
  k+1 -n, n+k+3, \frac{11}{12}, \frac{7}{12}  \\
 k+1+\frac{11}{12}, k+1+\frac{19}{12}, 2
     \end{array} \right| y  \right)\right]
\eg
\end{eqnarray*}
is
\begin{eqnarray*}
\bg
\frac{(k-n)_m(n+k+2)_m(-1/12)_m(7/12)_m}{(k+11/12)_m(k+19/12)_m(m!)^2}\\
-\frac{5m}{12}\frac{(k-n)(n+k+2)}{(k+19/12)(k+11/12)}\frac{(k-n+1)_{m-1}(n+k+3)_{m-1} (11/12)_{m-1}(7/12)_{m-1}}{(k+1+11/12)_{m-1}(k+1+19/12)_{m-1} (2)_{m-1}(m)!}.
\eg
\end{eqnarray*}
Using the identity $(z)_m=z(z+1)_{m-1}$ we get that coefficient to be
\bea
\bg
\frac{(k-n)_m(n+k+2)_m(-1/12)_m(7/12)_m}{(k+11/12)_m(k+19/12)_m(m!)^2}\left(1+\frac{(-5/12)(-12m)}{(7/12+m-1)}\right)\\
=\frac{(k-n)_m(n+k+2)_m(-1/12)_m(-5/12)_m}{(k+11/12)_m(k+19/12)_m(m!)^2}\left(1-\frac{72 m}{5}\right)\\
=\frac{1}{5}\frac{(k-n)_m(n+k+2)_m(-5/12)_m}{(k+11/12)_m(k+19/12)_m(m!)^2}\left(6(11/12)_m-(-1/12)_m\right).
\eg
\eea
In  the last equality we used
\[
(-1/12)_m(m-5/72)=(-1/12)_m[(m-1/12)+1/72]=\frac{-1}{12} (11/12)_m+\frac{1}{72}(-1/12)_m.
\]
It now follows that
\bea\label{4f3y}
\bg
 \left[{}_4F_3 \left(\left. \begin{array}{c}
  k -n, n+k+2, \frac{-1}{12}, \frac{7}{12}  \\
 k+\frac{11}{12}, k+\frac{19}{12}, 1
     \end{array} \right| y  \right)\right.\\
\left.-\frac{5y}{12}\frac{(k-n)(n+k+2)}{(k+19/12)(k+11/12)}{}_4F_3 \left(\left. \begin{array}{c}
  k+1 -n, n+k+3, \frac{11}{12}, \frac{7}{12}  \\
 k+1+\frac{11}{12}, k+1+\frac{19}{12}, 2
     \end{array} \right| y  \right)\right]\\
=\frac{6}{5}{}_4F_3 \left(\left. \begin{array}{c}
  k -n, n+k+2, \frac{11}{12}, \frac{-5}{12}  \\
 k+\frac{11}{12}, k+\frac{19}{12}, 1
     \end{array} \right| y  \right)-\frac{1}{5}{}_4F_3 \left(\left. \begin{array}{c}
  k -n, n+k+2, \frac{-1}{12}, \frac{-5}{12}  \\
 k+\frac{11}{12}, k+\frac{19}{12}, 1
     \end{array} \right| y  \right).
\eg
\eea
The result now follows by substituting \eqref{4f3y} with $y=1$ into \eqref{Bcoeffx} .
\end{proof}

\begin{rem}
We note that Kaneko and Zagier gave another explicit representation of a somehow different form than \eqref{ourrep} for the Atkin polynomials. Indeed, it follows form Theorem 4 (ii) in \cite{Kan:Zag} that
\be\label{kzrep}
\mathcal{A}_n(x)=\sum_{i=0}^n\sum_{m=0}^i (-1)^m \binom{\frac{-1}{12}}{i-m}\binom{\frac{-5}{12}}{i-m}\binom{n+\frac{1}{12}}{m}\binom{n-\frac{7}{12}}{m}\binom{2n-1}{m}^{-1} x^{n-i}.
\ee
\end{rem}

\section{Asymptotics} In the proof of Theorem 1 in \cite{Wim}, Wimp shows that the 
functions $u_n$ and $y_n$ ($u_n$ and $v_n$ in the notation of loc. cit.) defined by
\bea
\label{equn}
\bg
u_n^{(\al, \beta)}(x;c) = (-1)^n \frac{\Gamma(n+\beta+c+1)}{\Gamma(n+c+1)}
 F \left(\left. \begin{array}{c}
   -n-c, n+\al+  \beta+c+1 \\
 1+\beta
     \end{array} \right| x  \right), \\
 y_n^{(\al, \beta)}(x;c)  =(-1)^n \frac{\Gamma(n+ \al +c+1)}{\Gamma(n+\al+\beta+ c+1)}
 F \left(\left. \begin{array}{c}
   -n-\beta-c, n+\al+ c+1 \\
 1- \beta
     \end{array} \right| x  \right),   
\eg
\eea
satisfy the same recurrence relation satisfied by $R_n$ and $\mathcal{R}_n$, and thus the latter can be represented as linear combinations of the former. We shall slightly modify those functions so as to replace the gamma factors by rising factorials (and thus getting rational rather than transcendental coefficients when the parameters are rational) as follows. Set
\bea\label{UZ}
\bg
U_n^{(\al, \beta)}(x;c)=\frac{\Gamma(c+1)}{\Gamma(\beta+c+1)}u_n^{(\al, \beta)}(x;c),\\
Y_n^{(\al, \beta)}(x;c)= \frac{\Gamma(\alpha+\beta+c+1)}{\Gamma(\alpha+c+1)}y_n^{(\al, \beta)}(x;c). 
\eg
\eea
Thus we have
\bea
\label{eqUn}
\bg
U_n^{(\al, \beta)}(x;c)= (-1)^n \frac{(\beta+c+1)_n}{(c+1)_n}
 F \left(\left. \begin{array}{c}
   -n-c, n+\al+  \beta+c+1 \\
 1+\beta
     \end{array} \right| x  \right), \\
Y_n^{(\al, \beta)}(x;c)  =(-1)^n \frac{( \al +c+1)_n}{(\al+\beta+ c+1)_n}
 F \left(\left. \begin{array}{c}
   -n-\beta-c, n+\al+ c+1 \\
 1- \beta
     \end{array} \right| x  \right).   
\eg
\eea
Note that since the factors multiplied by $u_n$ and $y_n$ in \eqref{UZ} are independent of $n$, then $U_n$ and $Y_n$ satisfy the same recurrence as $R_n$ and $\mathcal{R}_n$. Indeed, after a simple Kummer transformation, formula (28) on p. 988 of \cite{Wim} can be written as
\bea\label{Ruv}
\bg
 R_n=\frac{(\beta+c)(\alpha+\beta+c)}{\beta(\alpha+\beta+2c)} F \left(\left. \begin{array}{c}
   c, 1-(\alpha+\beta+c) \\
 1- \beta
     \end{array} \right| x  \right)U_n\\
-\frac{c(\alpha+c)}{\beta(\alpha+\beta+2c)} F \left(\left. \begin{array}{c}
  \beta+c, 1-(\alpha+c) \\
 1+ \beta
     \end{array} \right| x  \right)Y_n.
\eg
\eea
Similarly,  Theorem 3.10 of \cite{Ism:Mas} leads to
\bea\label{calRuv}
\bg
\mathcal{R}_n= F \left(\left. \begin{array}{c}
   c, -(\alpha+\beta+c) \\
 - \beta
     \end{array} \right| x  \right)U_n\\
-\frac{c(\alpha+c)}{\beta(\beta+1)} xF \left(\left. \begin{array}{c}
  1+\beta+c, 1-(\alpha+c) \\
 2+ \beta
     \end{array} \right| x  \right)Y_n.
\eg
\eea
The following theorem provides the analogous representation for the Atkin polynomials.
\begin{thm}\label{theoremBuv}
Let $U_n$ and $Y_n$ be as in \eqref{eqUn} and set
\bea\label{umonic}
\bg
\tilde{U}_n^{(\al, \beta)}(x;c) = \frac{(c+1)_n(\alpha+\beta+c+1)_n}{(\alpha+\beta+2c+1)_{2n}} U_n^{(\al, \beta)}(x;c),\\
\tilde{Y}_n^{(\al, \beta)}(x;c) = \frac{(c+1)_n(\alpha+\beta+c+1)_n}{(\alpha+\beta+2c+1)_{2n}} Y_n^{(\al, \beta)}(x;c). 
\eg
\eea 
Then we have
\bea\label{Buv}
\mathcal{A}_{n+1}(x) = C(x) \tilde{U}_n^{(\frac{1}{2},\frac{-2}{3})}\left(x;\frac{7}{12}\right) + D(x) \tilde{Y}_n^{(\frac{1}{2},\frac{-2}{3})}\left(x;\frac{7}{12}\right),\,  n \ge 0,
\eea
with $C(x)$ and  $D(x)$ given by
\bea\label{CD}
\bg
C(x):=\frac{-1}{60}\left(24 F \left(\left. \begin{array}{c}
   \frac{-5}{12}, \frac{-5}{12} \\
 \frac{-1}{3}
     \end{array} \right| x  \right)+F \left(\left. \begin{array}{c}
   \frac{-5}{12}, \frac{-5}{12} \\
 \frac{2}{3}
     \end{array} \right| x  \right)\right),\\
D(x):=\frac{91}{384}x\left(4 F \left(\left. \begin{array}{c}
   \frac{-1}{12}, \frac{-1}{12} \\
 \frac{1}{3}
     \end{array} \right| x  \right)-5F \left(\left. \begin{array}{c}
   \frac{11}{12}, \frac{-1}{12} \\
 \frac{4}{3}
     \end{array} \right| x  \right)\right).
\eg
\eea
\end{thm}
\begin{proof}
From \eqref{Ruv} and \eqref{calRuv} we see that
\bea
\bg
xR_n^{(\frac{1}{2},\frac{-2}{3})}\left(x;\frac{7}{12}\right)-\frac{5}{12}\mathcal{R}_n^{(\frac{1}{2},\frac{-2}{3})}\left(x;\frac{7}{12}\right)=\\
 \frac{5}{12}U_n^{(\frac{1}{2},\frac{-2}{3})}\left(x;\frac{7}{12}\right)\left(\frac{(-1/12)}{(-2/3)}x F \left(\left. \begin{array}{c}
   \frac{7}{12}, \frac{7}{12} \\
 \frac{5}{3}
     \end{array} \right| x  \right)- F \left(\left. \begin{array}{c}
   \frac{7}{12}, \frac{-5}{12} \\
 \frac{2}{3}
     \end{array} \right| x  \right)\right)\\
-xY_n^{(\frac{1}{2},\frac{-2}{3})}\left(x;\frac{7}{12}\right)
\left(\frac{(7/12)(13/12)}{(-2/3)} F \left(\left. \begin{array}{c}
   \frac{-1}{12}, \frac{-1}{12} \\
 \frac{1}{3}
     \end{array} \right| x  \right)- \frac{5}{12}\frac{(7/12)(13/12)}{(-2/3)(1/3)}F \left(\left. \begin{array}{c}
   \frac{11}{12}, \frac{-1}{12} \\
 \frac{4}{3}
     \end{array} \right| x  \right)\right).\\
\eg
\eea
Expanding the hypergeometric series in powers of $x$, we get after some computation
\bea
\bg
xR_n^{(\frac{1}{2},\frac{-2}{3})}\left(x;\frac{7}{12}\right)-\frac{5}{12}\mathcal{R}_n^{(\frac{1}{2},\frac{-2}{3})}\left(x;\frac{7}{12}\right)=\\
\frac{-1}{60}\left(24 F \left(\left. \begin{array}{c}
   \frac{-5}{12}, \frac{-5}{12} \\
 \frac{-1}{3}
     \end{array} \right| x  \right)+F \left(\left. \begin{array}{c}
   \frac{-5}{12}, \frac{-5}{12} \\
 \frac{2}{3}
     \end{array} \right| x  \right)\right)U_n^{(\frac{1}{2},\frac{-2}{3})}\left(x;\frac{7}{12}\right)\\
+\frac{91}{384}x\left(4 F \left(\left. \begin{array}{c}
   \frac{-1}{12}, \frac{-1}{12} \\
 \frac{1}{3}
     \end{array} \right| x  \right)-5F \left(\left. \begin{array}{c}
   \frac{11}{12}, \frac{-1}{12} \\
 \frac{4}{3}
     \end{array} \right| x  \right)\right) Y_n^{(\frac{1}{2},\frac{-2}{3})}\left(x;\frac{7}{12}\right),
\eg
\eea
and the result follows from \eqref{Brep}.
\end{proof}
Theorem \ref{theoremBuv} enables us to obtain an asymptotic formula for the Atkin polynomials.
\begin{thm}
Let $C(x)$ and $D(x)$ be as in \eqref{CD}. For fixed $\theta\in\left(0,\frac{\pi}{2}\right)$, the following asymptotic formula holds as $n\to \infty$ 
\bea 
\bg
\mathcal{A}_{n+1}(\sin^2\theta)\sim\frac{(-1)^n}{2^{2n+1}(\cos\theta)(\sin\theta)^{\frac76}}\times\\
C(\sin^2\theta)\frac{\Gamma\left(\frac{1}{3}
\right)(\sin \theta)^{\frac{2}{3}}}{\Gamma\left(\frac{11}{12}
\right)\Gamma\left(\frac{17}{12}
\right)}\cos\left[2(n-1)\theta+\frac{\pi}{12}\right]\\
+D(\sin^2\theta)\frac{\Gamma\left(\frac{5}{3}\right)}{\Gamma\left(\frac{13}{12}
\right)\Gamma\left(\frac{19}{12}
\right)}\cos\left[2(n-1)\theta-\frac{7\pi}{12}\right]
\eg
\eea

\end{thm}
\begin{proof}
We start by recalling the following asymptotic formula due to Watson, \cite[(8) p. 237]{Luk}, (all of our asymptotic formulas will be as $n \to \infty$). 
 \bea
 \label{eqWatson}
 \begin{gathered}
 F \left(\left. \begin{array}{c}
 b  -n, n+a \\
d
     \end{array} \right| \sin^2 \theta   \right)   
     \sim\frac{\Gamma(d) n^{-d+\frac12}}{\sqrt{\pi}}  \; \frac{(\cos \theta)^{d-a-b-\frac12}}{(\sin \theta)^{d-\frac12}}
 \qquad     \\
\qquad \times    \; 
     \cos \left[2n\theta +(a-b) \theta -\frac{\pi}{2} \left(d-\frac{1}{2}\right)\right]\;
 \end{gathered}
 \eea
for fixed $\theta \in (0, \pi)$. 
Note that Stirling's formula can be written as
\be\label{Stirling}
\Gamma(n+a)\sim \sqrt{2\pi}\, n^{n+a-\frac12} e^{-n} \textrm{ as } n \to \infty,
\ee
from which we deduce
\[
\frac{\Gamma(n+a)}{\Gamma(n+b)}\sim n^{a-b}.
\]
Hence
\[
u_n(\sin^2\theta)\sim   \frac{(-1)^n\Gamma(1+\beta)(\cos\theta)^{-\alpha-\frac{1}{2}}}{\sqrt{\pi n} (\sin\theta)^{\beta+\frac{1}{2}}}\cos\left[2n\theta-(\alpha+\beta+2c+1)\theta-\frac{\pi}{2}\left(\beta+\frac{1}{2}\right)\right], 
\]

\[
y_n(\sin^2\theta)\sim   \frac{(-1)^n\Gamma(1-\beta)(\cos\theta)^{-\alpha-\frac{1}{2}}}{\sqrt{\pi n} (\sin\theta)^{-\beta+\frac{1}{2}}}\cos\left[2n\theta-(\alpha+\beta+2c+1)\theta+\frac{\pi}{2}\left(\beta-\frac{1}{2}\right)\right]. 
\]
Also,  from \eqref{Stirling} we get
\bea
\bg
\frac{(c+1)_n(\alpha+\beta+c+1)_n}{(\alpha+\beta+2c+1)_{2n}}=\frac{\Gamma(\alpha+\beta+2c+1)}{\Gamma(c+1)\Gamma(\alpha+\beta+c+1)}\frac{\Gamma(n+c+1)\Gamma(n+\alpha+\beta+c+1)}{\Gamma(2n+\alpha+\beta+2c+1)}\\
\sim \frac{\Gamma(\alpha+\beta+2c+1)}{\Gamma(c+1)\Gamma(\alpha+\beta+c+1)}\sqrt{\pi n}\left(\frac{1}{2}\right)^{2n+\alpha+\beta+2c}.
\eg
\eea
Substituting $(\alpha, \beta,c)=\left(\frac12, \frac{-2}3,\frac7{12}\right)$ we see that
\bea
\bg
\tilde{U}_n(\sin^2\theta)\sim \frac{(-1)^n\Gamma\left(\frac{1}{3}\right)(\sin \theta)^{\frac16}}{2^{2n+1}(\cos\theta) \Gamma\left(\frac{11}{12}\right)\Gamma\left(\frac{17}{12}\right)}\cos\left[2(n-1)\theta+\frac{\pi}{12}\right],\\
\tilde{Y}_n(\sin^2\theta)\sim \frac{(-1)^n\Gamma\left(\frac{5}{3}\right)(\sin \theta)^{\frac{-7}6}}{2^{2n+1}(\cos\theta) \Gamma\left(\frac{13}{12}\right)\Gamma\left(\frac{19}{12}\right)}\cos\left[2(n-1)\theta-\frac{7\pi}{12}\right],
\eg
\eea
and the result follows from \eqref{Buv} and \eqref{CD}.
\end{proof}

\section{Generating Functions}
We start by recalling a remarkable identity of Flensted-Jensen and Koornwinder \cite{Flen:Koor}. The interested reader could also consult \cite{Wim} for more details on various other authors who presented variants of this identity as well as other proofs.
\begin{lemma}
Let $t, x, a, b, d$ be complex numbers with $x\notin [1, \infty)$ and 
\be\label{tx}
|t|<\frac{1}{|\sqrt{x}+\sqrt{x-1}|^2}.
\ee
Then
\bea\label{nicesum1}
\bg
\sum_{n=0}^\infty \frac{(d+a)_n (b)_n}{(a+b+1)_n} F\left(\left. \ba{c} -n-a,  n+b\\
 d\ea \right|x\right)\frac{(-t)^n}{n!}\\
=\left(\frac{z_2-t}{z_2+t}\right)^{a+d}\left(\frac{2}{z_2-t}\right)^b F\left(\left. 
\ba{c} -a,  b\\
 d\ea \right| \frac{t+z_1}{2t} \right)F\left(\left. \ba{c} a+d,  a+1\\
 a+b+1\ea \right| \frac{2t}{t+z_2} \right)
\eg
\eea
where $z_1=1-\sqrt{(1+t)^2-4xt}$ and $z_2=1+\sqrt{(1+t)^2-4xt}$.
\end{lemma}
To simplify notation we shall write, for $t\neq 0$,
\bea\label{delta}
\bg
\delta=\frac{t+z_1}{2t}=\frac{(1+t)-\sqrt{(1+t)^2-4xt}}{2t},\\
\epsilon=\frac{t+z_2}{2t}=\frac{(1+t)+\sqrt{(1+t)^2-4xt}}{2t}.\\
\eg
\eea
We clearly have
\be
t(y-\delta)(y-\epsilon)=ty^2-(1+t)y+x.
\ee
Obviously $z_2+t=2t\epsilon$ and we also have $z_2-t=2t(\epsilon-1)$. Furthermore we have $\delta\epsilon=\frac{x}{t}$. Thus we can re-write \eqref{nicesum1} for $x\neq 0$ as
\bea\label{nicesum}
\bg
\sum_{n=0}^\infty \frac{(d+a)_n (b)_n}{(a+b+1)_n} F\left(\left. \ba{c} -n-a,  n+b\\
 d\ea \right|x\right)\frac{(-t)^n}{n!}\\
=\frac{(x-t\delta)^{a+d-b} \delta^b}{x^{a+d}}  
F\left(\left. \ba {c} -a,  b\\
 d  \ea \right|  \delta\right)
 F\left(\left. \ba{c} a+d,  a+1\\     a+b+1 \ea \right|  \frac{t}{x}\delta\right).   
\eg
\eea

The following proposition  provides a  generating function for $U_n$ and $Y_n$. 
\begin{prop}\label{genUY}
Let $U_n$ and $Y_n$ be as in \eqref{UZ}. For $t$ and $x$ such that $0\neq x\notin[1,\infty)$ and $|t(\sqrt{x}+\sqrt{x-1})^2|<1$, and set $\delta$ as in \eqref{delta}. Then the following identities hold.
\bea\label{genun}
\bg
\sum_{n=0}^\infty \frac{(\alpha+\beta+c+1)_n (c+1)_n}{(\alpha+\beta+2c+2)_n} U_n(x) \frac{t^n}{n!}\\
=\frac{\delta^{\alpha+\beta+c+1}}{x^{\beta+c+1}(x-t\delta )^{\alpha}}F \left(\left. \ba{c}
   -c, \al+  \beta +c +1 \\
1+\beta
     \ea \right|  \delta  \right) F \left(\left.  \ba {c}
   \beta+c+1, c+1 \\
\alpha+\beta+2c+2
     \ea \right|  \frac{t\delta}{x}  \right),
\eg
\eea

\bea\label{genvn}
\bg
\sum_{n=0}^\infty  \frac{(\alpha+\beta+c+1)_n (c+1)_n}{(\alpha+\beta+2c+2)_n} Y_n(x) \frac{t^n}{n!}\\
=\frac{\delta^{\alpha+c+1}}{x^{c+1}(x-t\delta)^{\alpha}}
F \left(\left.  \ba{c}
  -\beta-c, \alpha+c+1 \\
1-\beta
\ea \right|  \delta  \right) F \left( \left. \ba {c}
    c+1,\beta+c+1 \\
\alpha+\beta+2c+2
    \ea \right| \frac{t\delta}{x}  \right).
\eg
\eea
\end{prop}

\begin{proof}
From \eqref{equn} we see that
\bea\label{sumun1}
\bg
\frac{\Gamma(c+1)}{\Gamma(\beta+c+1)}\frac{(\alpha+\beta+c+1)_n(c+1)_n}{(\alpha+\beta+2c+2)_n}u_n\\
=\frac{(c+\beta+1)_n(\alpha+\beta+c+1)_n}{(\alpha+\beta+2c+2)_n} F \left(\left. \begin{array}{c}
   -n-c, n+\al+  \beta+c+1 \\
1+\beta
     \end{array} \right| x  \right),\\
\eg
\eea
and
\bea\label{sumvn1}
\bg
\frac{\Gamma(\alpha+\beta+c+1)}{\Gamma(\alpha+c+1)}\frac{(\alpha+\beta+c+1)_n(c+1)_n}{(\alpha+\beta+2c+2)_n}y_n\\
=\frac{(c+1)_n(\alpha+c+1)_n}{(\alpha+\beta+2c+2)_n} F \left(\left. \begin{array}{c}
   -n-\beta-c, n+\al+  c+1 \\
1-\beta
     \end{array} \right| x  \right).\\
\eg
\eea
The identities \eqref{genun} and \eqref{genvn} follow from applying \eqref{nicesum} with the choices  $(a,b,d)=(c, \alpha+\beta+c+1, \beta+1)$ and $(a,b,d)=(\beta+c, \alpha+c+1, 1-\beta)$, respectively.
\end{proof}
\begin{rem}
The result in Proposition \ref{genUY} is essentially due to Wimp. However, we take this opportunity to correct a misprint in  the statement of Theorem 5 in \cite{Wim}: In  the first line of page 999,  the parameter ``$\gamma+c+\beta$" should be replaced by ``$\gamma+c-\beta"$ (in our notation, the later is $\alpha+c+1$ while the former would be $\alpha+c+1+2\beta$, which indeed doesn't ever seem to figure in the theory).  
\end{rem}

We next obtain a generating function identity for the Atkin polynomials scaled by a rather unexpected appearance of the Catalan numbers. The right hand side of the generating series has four summands; each is up to relatively simple multiple a product of three hypergeometric functions in the variables $x$, $\delta$ and $\frac{1}{\epsilon}=\frac{t\delta}{x}$. 
\begin{thm}
Let $C(x) $ and $D(x)$ be as in \eqref{CD}, and $\delta$ as in \eqref{delta}. Furthermore, let $\{C_n=\frac{1}{n+1}\binom{2n}{n}\}_n$ denote the sequence of Cataln numbers.

\begin{enumerate}

\item    For $0< x<1$ and $|t|<1$ we have
 
\bea\label{Bgen}
\bg
\sum_{n=0}^\infty C_{n+1}\mathcal{A}_{n+1}(x)t^n  =  \\
\frac{\delta^{\frac{17}{12}}}{x^{\frac{11}{12}}\sqrt{x-t\delta}}
 F \left( \left. \ba{c}
   \frac{11}{12}, \frac{19}{12} \\
3
 \ea \right| \frac{t\delta}{x}  \right) \\
\times\left[C(x)F \left(\left. \ba {c}
  \frac{-7}{12},  \frac{17}{12} \\
\frac{1}{3}
 \ea \right|  \delta  \right)
+D(x) \left(\frac{x}{\delta}\right)^{\frac{2}{3}}F \left(\left. \ba {c}
  \frac{1}{12},  \frac{25}{12} \\
\frac{5}{3}
 \ea \right|  \delta  \right)
\right]
\eg
\eea

\item For $|t|<1$ we have
\be\label{gen0}
\sum_{n=0}^\infty C_{n+1}\mathcal{A}_{n+1}(0)(-t)^n=\frac{-5}{12}
 F \left( \left. \ba{c}
   \frac{11}{12}, \frac{17}{12} \\
3
 \ea \right|t  \right), \\
\ee
and consequently, for $n\geq 0$ we have
\be\label{B0}
\mathcal{A}_{n+1}(0)=(-1)^n \left(\frac{-5}{12}\right) \frac{\left(\frac{11}{12}\right)_n\left(\frac{17}{12}\right)_n}{(2n+1)!}.
\ee
\end{enumerate}
\end{thm}

\begin{proof}
Note that for $0\leq x<1$, we have 
\[
|\sqrt{x}+\sqrt{x-1}|^2=|\sqrt{x}+i\sqrt{1-x}|^2=1,
\] 
so \eqref{tx} indeed translates into $|t|<1$. 
Now using \eqref{umonic} we see that
\bea
\bg
 \frac{(\alpha+\beta+c+1)_n (c+1)_n}{(\alpha+\beta+2c+2)_n} U_n = \frac{(\alpha+\beta+2c+1)_{2n}}{(\alpha+\beta+2c+2)_n} \tilde{U}_n\\
=(\alpha+\beta+2c+1) \frac{(\alpha+\beta+2c+1+n)_{n}}{(\alpha+\beta+2c+1+n)} \tilde{U}_n, 
\eg
\eea
with a similar identity for $Y_n$, and \eqref{Bgen} now follows from \eqref{Buv} and Proposition \ref{genUY} by substituting $(\alpha, \beta, c)=\left(\frac{1}{2}, \frac{-2}{3}, \frac{7}{12}\right)$.

When $x=0$ and $|t|<1$, then, in the notation of \eqref{nicesum1}, we have $t+z_1=0$ and $z_2+t=2(1+t)$, and hence $z_2-t=2$. Furthermore we have $C(0)=\frac{-5}{12}$ and $D(0)=0$. It thus follows from \eqref{nicesum1} and \eqref{Buv} that
\be\label{gen0t}
2\sum_{n=0}^\infty \binom{2n+1}{n}\mathcal{A}_{n+1}(0)\frac{t^n}{2+n}=\frac{-5}{12(1+t)^{\frac{11}{12}}}  
 F \left( \left. \ba{c}
   \frac{11}{12}, \frac{19}{12} \\
3
 \ea \right| \frac{t}{1+t}  \right). \\
\ee
Replacing $t$ with $(-t)$ and applying Pfaff-Kummer transformation (formula (2) on p. 105 of \cite{Erd:Mag:Obe:Tri}), we  obtain \eqref{gen0}, from which \eqref{B0} follows by comparing coefficients and simplifying.

\end{proof}

\begin{rem}
Formula \eqref{B0} can also be obtained directly from the defining recursion of the Atkin polynomials as in Proposition 6 of \cite{Kan:Zag}. In that same proposition, and again using only the defining recurrence \eqref{rec1}, Kaneko and Zagier also obtain a formula equivalent to
\be\label{B1}
\mathcal{A}_{n+1}(1)=\frac{7}{12}\frac{\left(\frac{11}{12}\right)_n\left(\frac{19}{12}\right)_n}{(2n+1)!}.
\ee
Taking a hint from \eqref{gen0}, it is straightforward to prove directly from \eqref{B1} that for $|t|<1$ we have
\be\label{gen1}
\sum_{n=0}^\infty C_{n+1}\mathcal{A}_{n+1}(1)t^n=\frac{7}{12}
 F \left( \left. \ba{c}
   \frac{11}{12}, \frac{19}{12} \\
3
 \ea \right|t  \right). \\
\ee
Alternatively one can prove \eqref{gen1} in a manner similar to \eqref{gen0}, bearing in mind that we have $C(1)=D(1)=0$ whereas $ \tilde{U}_n^{(\frac{1}{2},\frac{-2}{3})}\left(x;\frac{7}{12}\right)$ and $ \tilde{V}_n^{(\frac{1}{2},\frac{-2}{3})}\left(x;\frac{7}{12}\right)$ have simple poles at $x=1$, and thus their product is to be interpreted in the limit as $x\to 1^-$ as the derivative of the former multiplied by the residue of the latter.
\end{rem}

\section{The weight function for the Atkin Polynomials}
In \cite{Kan:Zag}, Kaneko and Zagier gave the weight function for the Atkin polynomials $A_n(j)$ on $[0,1728]$ as 
\be
w(j)=\frac{6}{\pi}\theta^\prime(j),
\ee
where $\theta:[0,1728]\rightarrow \left[\frac{\pi}{3}, \frac{\pi}{2}\right]$ is the inverse of the monotone increasing 
 function $\theta\mapsto j(e^{i\theta})$, where $j(\tau)$ is the usual modular $j$-invariant 
 from the theory of modular forms.
In this section we derive an explicit description of the weight function in terms of hypergeometric series.
Formula (25) on p. 20 of \cite{Erd:Mag:Obe:Tri} states that an inverse for the scaled $j$-invariant given by 
\[
J(z)=\frac{j(z)}{1728},
\]
is obtainable by the formula
\be\label{25}
z=e^{2\pi i/3}\frac{F-\lambda e^{i\pi/3}J^{1/3}F^*}{F-\lambda e^{-i\pi/3}J^{1/3}F^*},
\ee
where
\be\label{lam}
\bg
F(J)=\, _2F_1\left(\begin{matrix} \frac{1}{12} & \frac{1}{12}\\
 &\frac{2}{3}\end{matrix} ; J\right),\\
F^*(J)=\, _2F_1\left(\begin{matrix} \frac{5}{12} & \frac{5}{12}\\
 &\frac{4}{3}\end{matrix} ; J\right),\\
\la=\frac{\Gamma(2/3)\Gamma(5/12)\Gamma(11/12)}{\Gamma(4/3)\Gamma(1/12)\Gamma(7/12)}\\
=(2-\sqrt{3})\frac{\Gamma(2/3)\Gamma^2(11/12)}{\Gamma(4/3)\Gamma^2(7/12)}.
\eg
\ee
We must note that this is one inverse of many as $J$ is invariant under modular transformations. This particular formula gives, easily, that $z(0)=e^{2\pi i/3}$. In order to use the same intervals as in \cite{Kan:Zag}, we consider another inverse corresponding to applying $z\mapsto \frac{-1}{z}$; thus obtaining
\be\label{Jinv}
z(J)=e^{\pi i/3}\frac{F-\lambda e^{-i\pi/3}J^{1/3}F^*}{F-\lambda e^{i\pi/3}J^{1/3}F^*}.
\ee
It is straightforward to verify that using \eqref{Jinv}, we get $z(0)=e^{\frac{\pi i}{3}}$ and $z(1)=e^{\frac{\pi i}{2}}$.
For $0\leq J\leq 1$, $F$ and $F^*$ are computed in terms of the converging hypergeometric series and hence are  real. Thus in the ratio
\[
\frac{F(J)-\lambda e^{-i\pi/3}J^{1/3}F^*(J)}{F(J)-\lambda e^{i\pi/3}J^{1/3}F^*(J)}
\]
the denominator is the complex conjugate of the numerator. Hence the ratio has absolute value equal to $1$, and is of the form $e^{i\rho}$. We will show below that  $0\leq \rho \leq \pi/6$. Thus an explicit description of the function $\theta(j):[0,1728]\rightarrow [\pi/3, \pi/2]$ is given by $\theta(j)=\phi(\frac{j}{1728})$ where $\phi(J):[0,1]\rightarrow \left[\frac{\pi}{3}, \frac{\pi}{2}\right]$ is defined by
\[
\phi(J)=\frac{\pi}{3} -i\log\left(\frac{F(J)-\lambda e^{-i\pi/3}J^{1/3}F^*(J)}{F(J)-\lambda e^{i\pi/3}J^{1/3}F^*(J)}\right)=\frac{\pi}{3}+\rho(J),
\]
and we have
\bea
\bg
\phi^\prime(J)=-i\frac{F(J)-\lambda e^{i\pi/3}J^{1/3}F^*(J)}{F(J)-\lambda e^{-i\pi/3}J^{1/3}F^*(J)}\left(\frac{F(J)-\lambda e^{-i\pi/3}J^{1/3}F^*(J)}{F(J)-\lambda e^{i\pi/3}J^{1/3}F^*(J)}\right)^\prime\\
=-i\frac{W(J)}{|F(J)-\lambda e^{-i \pi/3}J^{1/3}F^*(J)|^2},
\eg
\eea
where $W(J)$ is given explicitly by
\be\label{W2}
\bg
W(J)=(F(J)-\la e^{i\pi/3}J^{1/3}F^*(J))(F^\prime(J)-\la e^{-i\pi/3}J^{1/3}(F^*)^\prime(J)-\frac{\la}{3}e^{-i\pi/3}J^{-2/3}F^*(J))\\
-(F(J)-\la e^{-i\pi/3}J^{1/3}F^*(J))(F^\prime(J)-\la e^{i\pi/3}J^{1/3}(F^*)^\prime(J)-\frac{\la}{3}e^{i\pi/3}J^{-2/3}F^*(J))\\
=\frac{\la}{3}J^{-2/3}i\sqrt{3}\left((F(J)F^*(J)+3JF(J)(F^*)^\prime(J)-3JF^\prime(J)F^*(J)\right)
\eg
\ee
We also note that $W$ is the Wronskian of two linearly independent solutions for the equation
\[
z(1-z)\frac{d^2u}{dz^2}+(c-(1+a+b)z)\frac{du}{dz}-ab u=0,
\]
where here $a=b=\frac{1}{12}$ and $c=\frac{2}{3}$. It follows that $W$ itself satisfies the equation
\be\label{diffeqW}
z(1-z)\frac{d W}{dz}=((a+b+1)z-c) W.
\ee
On the open interval $(0,1)$, \eqref{diffeqW} has solution

\bea\label{W1}
W(J)=B J^{-2/3} (1-J)^{-1/2}.
\eea
To determine the constant $B$ we compare the coefficient of $J^{-2/3}$ in  \eqref{W1} and \eqref{W2} to get
\[
B=\frac{i\la}{\sqrt{3}}.
\] 
Hence

\be
\phi^\prime(J)=\frac{\la}{\sqrt{3}}\frac{J^{-2/3}(1-J)^{-1/2}}{|F(J)-\lambda e^{-i \pi/3}J^{1/3}F^*(J)|^2}.
\ee
The fact that the derivative is positive for $0\leq J\leq 1$ implies that $\phi(J)$ is monotone increasing, and hence that it is bounded between $\phi(0)$ and $\phi(1)$, as we claimed above. 

Note that
\bea
\bg
w(j)=\frac{6}{\pi}\theta^\prime(j)=\frac{6}{1728 \pi}\phi^\prime\left(\frac{j}{1728}\right)\\
=\frac{6\la}{1728\pi\sqrt{3}} \frac{12(12^2 j^{-2/3})( (1728-j)^{-1/2}12^{3/2})}{\left|12F\left(\frac{j}{1728}\right)-\lambda e^{-i \pi/3}j^{1/3}F^*\left(\frac{j}{1728}\right)\right|^2}.\\
\eg
\eea
We have thus proved the following theorem.

\begin{thm}
Let $\la$ be as in \eqref{lam}. Then the normalized weight function for the Atkin polynomials $A_n(j)$ on the interval $[0,1728]$ is given by
\be\label{thmwj}
w(j)= \frac{144\la}{\pi}\frac{ j^{-2/3} (1728-j)^{-1/2}}{\left|12F\left(\begin{matrix} \frac{1}{12} & \frac{1}{12}\\  &\frac{2}{3}\end{matrix};\frac{j}{1728}\right)-\lambda e^{-i \pi/3}j^{1/3}F\left(\begin{matrix} \frac{5}{12} & \frac{5}{12}\\  &\frac{4}{3}\end{matrix};\frac{j}{1728}\right)\right|^2}.\\
\ee
\end{thm}




\vspace{.2in}

{\bf Acknowledgments}: Mourad Ismail wishes to acknowledge the hospitality of Texas A$\&$M University Qatar  during his visit in September 2012 when this work started. Both 
authors thank Richard Askey, Jet Foncannon (Wimp), Masanobu Kaneko, and Don Zagier for their  interest in this work and for very helpful hints. They are also grateful for the referee for constructive remarks and for suggesting a number of additional references related to the paper.


\noindent{Ahmad El-Guindy}\\
{Current address: Science Program, Texas A\&M University in Qatar, Doha, Qatar}\\
{Permanent address: Department of Mathematics, Faculty of Science, Cairo University, Giza, Egypt 12613}\\
email: {a.elguindy@gmail.com}
  \bigskip
  
\noindent M. E. H. I,
  University of Central Florida, 
Orlando, Florida 32828, \\
and King Saud University,  Riyadh, Saudi Arabia\\
  email: mourad.eh.ismail@gmail.com


\begin{thebibliography}{99}

%

\bibitem{And:Ask:Roy}  G. E. Andrews, R. A. Askey, and R. Roy,
{\it Special Functions}, Cambridge University Press, Cambridge, 1999.

\bibitem{BrilMort} J. Brillhart and P. Morton, Class numbers of quadratic fields, Hasse invariants of elliptic curves,
and the supersingular polynomial, J. Number Theory {\bf 106} (2004), no. 1,  79--111.

\bibitem{Chi} T. S. Chihara,
  {\it An Introduction to Orthogonal Polynomials}, Gordon and Breach,
New York, 1978, reprinted by Dover, New York, 2010.

\bibitem{Erd:Mag:Obe:Tri}
 A. Erd{\'{e}}lyi,  W. Magnus,  F. Oberhettinger and F. G. Tricomi,  {\it Higher Transcendental Functions},
volume   1, McGraw-Hill,  New York,   1953
%
\bibitem{Flen:Koor} M. Flensted-Jensen and T. Koornwider,
The convolution structure for Jacobi function expansions, Ark. Mat. {\bf 11} (1975), 245--262.
%

%
%
\bibitem{Ism} M. E. H. Ismail, {\it Classical and Quantum Orthogonal Polynomials in one variable}, 
Cambridge University Press, paperback edition, Cambridge, 2009.
%
\bibitem{Ism:Let:Val} M. E. H. Ismail, J. Letessier and G. Valent, Linear birth and death models and 
 associated Laguerre polynomials, J. Approx. Theory {\bf 55} (1988), 337--348.
%
%
\bibitem{Ism:Mas} M. E. H. Ismail and D. R. Masson,  Two families of orthogonal polynomials 
related to Jacobi polynomials, Rocky Mountain J. Math.  {\bf 21} (1991),  , 359--375.
%
\bibitem{Jon:Thr}  W. B. Jones and W. Thron,
  {\it Continued Fractions: Analytic Theory and Applications}, 
 Addison-Wesley,
 Reading, MA,
  1980, republished by Cambridge University Press, Cambridge. 

%
\bibitem{Kan:Zag} M. Kaneko and D. Zagier,  Supersingular $j$-invariants, hypergeometric series, and Atkin's orthogonal polynomials, 
in ``Proceedings of the Conference on Computational Aspects of Number Theory",  AMS/IP Studies in Advanced Math. 7, International Press, Cambridge (1997) 97-126. 
%

\bibitem{Lane} M. E. Lane, Generalized Atkin Polynomials and Non-Ordinary Hyperelliptic Curves. Doctoral
Dissertation. University of California, Los Angeles. 2012.

\bibitem{Lor:Wad} L. Lorentzen and H. Waadeland, {\it 
Continued Fractions with Applications}, Northholand, Amsterdam, 1992. 
%
\bibitem{Luk} Y. Luke, {\it The Special Functions and Their Approximations}, volume 1, Academic Press, New York, 1969. 
%
\bibitem{Nev} P. Nevai, {\it Orthogonal polynomials}, Memoirs American Math. Soc. Number 213,  1979. 
%
%
\bibitem{Rai}  E. D. Rainville,   {\it Special Functions},  Macmillan ,
 New York,   1960.

%

\bibitem{Sak} Y. Sakai, The Atkin orthogonal polynomials for the low-level Fricke Groups and their
applications , Int. J. Number Theory {\bf 7} (2011), no. 6,  1637--1661.

%
\bibitem{Sze} G. Szeg\H{o}, {\it Orthogonal Polynomials}, Fourth edition. American Mathematical Society, Colloquium Publications, Vol. XXIII. American Mathematical Society, Providence, R.I., 1975. 

%
\bibitem{Tsu} H. Tsutsumi, The Atkin orthogonal polynomials for congruence subgroups of low levels,
Ramanujan J. {\bf 14} (2007), no. 2,  223--247.
%
\bibitem{Wim} J. Wimp (Foncannon), Explicit formulas for associated Jacobi polynomials, Canadian J. Math. 
{\bf 39} (1987), 983--1000. 


\end{thebibliography}
\end{document}